\documentclass[a4paper, 11 pt]{article}

\usepackage{amsfonts,amsmath,amssymb,hhline,ifthen,amsthm,graphics,cite,latexsym, caption2}
\usepackage{latexsym}
\usepackage[T2A]{fontenc}
\usepackage[cp1251]{inputenc}
\usepackage[english]{babel}
\usepackage[matrix, arrow, curve]{xy}
\usepackage{graphicx}

\newcounter{q3}
\newcommand{\dsm}[3]{
{\if#20{\if#31{\frac{\partial #1}{\partial y}}\else
          {\frac{\partial^{#3} #1}{\partial y^{#3}}}
        \fi}\else
  {\if#30{\if#21{\frac{\partial #1}{\partial x}}\else
            {\frac{\partial^{#2} #1}{\partial x^{#2}}}
          \fi}\else
    {\setcounter{q3}{#2}\addtocounter{q3}{#3}
    \frac{\partial{\if{1}\arabic{q3}^{}\else^{ \arabic{q3} }\fi}#1}
    {{\if#20\else{\partial x{\if#21\else{^{#2}}\fi}}\fi}
     {\if#30\else{\partial y\if#31\else{^{#3}}\fi}\fi} }}
   \fi}
\fi} }

\newcommand{\dzm}[2]{
{\if#10{\if#21{\frac{\partial}{\partial\overline{\zeta}}}\else
          {\frac{\partial^{#2}}{\partial\overline{\zeta}^{#2}}}
        \fi}\else
  {\if#20{\if#11{\frac{\partial}{\partial\zeta}}\else
            {\frac{\partial^{#1}}{\partial\zeta^{#1}}}
          \fi}\else
    {\setcounter{q3}{#1}\addtocounter{q3}{#2}
    \frac{\partial{\if{1}\arabic{q3}^{}\else^{ \arabic{q3} }\fi}}
    {{\if#10\else{\partial\zeta{\if#21\else{^{#1}}\fi}}\fi}
     {\if#20\else{\partial\overline{\zeta}\if#21\else{^{#2}}\fi}\fi} }}
   \fi}
\fi} }

\newcounter{q2}

\newtheoremstyle{theor}
  {\medskipamount}
  {\medskipamount}
  {\itshape}
  {}
  {\bfseries}
  {.}
  {.5em}
  {}

\newtheorem{definition}{Definition}[section]
\newtheorem{theorem}[definition]{Theorem}
\newtheorem{lemma}[definition]{Lemma}

\newtheorem{corollary}[definition]{Corollary}

\theoremstyle{definition}
\newtheorem{remark}[definition]{Remark}
\newtheorem{example}[definition]{Example}

\numberwithin{equation}{section}

\makeindex

\newtheoremstyle{remarks}
  {0mm}
  {0mm}
  {\itshape}
  {}
  {\itshape}
  {.}
  {.5em}
  {}
\makeatletter \@addtoreset{equation}{section} \makeatother

\begin{document}

\subsection*{\center TWO IDEALS CONNECTED WITH STRONG RIGHT UPPER POROSITY AT A POINT }\begin{center}\textbf{V. Bilet, O. Dovgoshey, J. Prestin} \end{center}
\parshape=5
1cm 13.5cm 1cm 13.5cm 1cm 13.5cm 1cm 13.5cm 1cm 13.5cm \noindent \small {\bf Abstract.} Let $\textbf{\emph{SP}}$ be the set of upper strongly porous at $0$ subsets of $\mathbb R^{+}$ and let $\hat I(\textbf{\emph{SP}})$ be the intersection of maximal ideals $\textbf{\emph{I}}\subseteq\textbf{\emph{SP}}$. Some characteristic properties of sets $E\in\hat I(\textbf{\emph{SP}})$ are obtained. It is shown that the ideal generated by the so-called  completely strongly porous at $0$ subsets of $\mathbb R^{+}$ is a proper subideal of $\hat I(\textbf{\emph{SP}}).$

\bigskip

\parshape=5
1cm 13.5cm 1cm 13.5cm 1cm 13.5cm 1cm 13.5cm 1cm 13.5cm \noindent \small {\bf Key words:} one-side porosity, local strong upper porosity, completely strongly porous set, ideal.

\parshape=2
1cm 15.5cm 1cm 15.5cm  \noindent \small {\bf }

 \bigskip
\textbf{ AMS 2010 Subject Classification: } 28A10, 28A05

\section {Introduction } \hspace*{\parindent}
The basic ideas concerning the notion of set porosity for the first time appeared in some early works of Denjoy \cite{D1},  \cite{D2} and Khintchine  \cite{Kh} and then arose independently in the study of cluster sets in 1967 (Dol\v{z}enko \cite{Dol1}).
A useful collection of facts related to the notion of porosity can be found, for example, in \cite{H}, \cite{HV}, \cite{Th} and \cite{Tk}. The porosity appears naturally in many problems and plays an implicit role in various arias of analysis (e. g., the cluster sets \cite{Z}, the Julia sets \cite{PR}, the quasisymmetric maps \cite{Va},  the differential theory \cite{KPS}, the theory of generalized subharmonic functions \cite{DRih} and so on). The reader can also consult \cite{Z3} and \cite{Z4} for more information.

The porosity found interesting applications in connection with ideals of sets. Well-known results for ideals of compact sets can be found, for example, in \cite{Kech} and \cite{KLW}. In many papers the authors investigate different characteristics (set-theoretic, descriptive, analytic) of the ideals of porous sets (see, e. g., \cite{Rep}, \cite{ZZ1}, \cite{ZZ2}). Some questions related to the order isomorphism between the principal ideals of porous sets of $\mathbb R$ were studied in \cite{SF}. Our paper is also a contribution to this line of research, in particular, we investigate two ideals, whose elements are upper strongly porous at $0$ subsets of $\mathbb R^{+}.$

\section{Right upper porosity at a point}
\hspace*{\parindent}
Let us recall the definition of the right upper porosity at a point. Let $E$ be a subset of $\mathbb R^{+}=[0, \infty).$

\begin{definition}\label{D1.1}
The right upper porosity of $E$ at $0$ is the number
\begin{equation}\label{E1.1}
p^{+}(E,0):=\limsup_{h\to 0^{+}}\frac{\lambda(E,0,h)}{h}
\end{equation}
where $\lambda(E,0,h)$ is the length of the largest open subinterval of $(0,h),$ which could be the empty set $\varnothing,$ that
contains no point of $E$. The set $E$ is porous on the right at $0$ if $p^{+}(E, 0)>0$ and $E$ is strongly porous on the right at $0$ if $p^{+}(E,0)=1.$
\end{definition}

For the remaining of the paper, when the porosity is considered, this will always be assumed to be the right upper porosity at $0.$

\medskip

For $E\subseteq\mathbb R^{+}$ define the subsets $\tilde E$ and  $\tilde H(E)$ of the set of sequences $\tilde h=\{h_n\}_{n\in\mathbb N}$ with $h_n\downarrow 0$ by the rules:
\begin{equation}\label{E1.2*}
(\tilde h\in\tilde E)\Leftrightarrow \left ( h_n\in E\setminus\{0\} \,\, \mbox{for all}\, n\in\mathbb N\right),
\end{equation}
and
\begin{equation}\label{E1.2}
(\tilde h\in\tilde H(E))\Leftrightarrow \left ( \frac{\lambda(E,0,h_n)}{h_n}\rightarrow
1 \quad \mbox{with} \quad n\rightarrow\infty\right),
\end{equation}
where the number $\lambda(E,0, h_n)$ is the same as in
Definition~\ref{D1.1}.

Define also an \emph{equivalence relation} $\asymp$ on the set of sequences of positive
numbers as follows. Let $\tilde a=\{a_n\}_{n\in\mathbb N}$ and
$\tilde{\gamma}=\{\gamma_n\}_{n\in\mathbb N}.$ Then $\tilde a \asymp \tilde {\gamma}$ if
there are positive constants $c_1$ and $c_2
$ such that
\begin{equation*}\label{equiv1}
c_1 a_n \le \gamma_n \le c_2 a_n
\end{equation*}
for all $n\in\mathbb N.$

\begin{definition}\label{D1.2}
Let $E\subseteq\mathbb R^{+}.$ The set $E$ is completely strongly porous on the right at $0$ if for every $\tilde \tau=\{\tau_n\}_{n\in\mathbb N}\in\tilde E$ there is $\tilde h=\{h_n\}_{n\in\mathbb N}\in\tilde H(E)$ such that $\tilde\tau\asymp\tilde h.$
\end{definition}

In what follows we denote by $\textbf{\emph{SP}}$ and $\textbf{\emph{CSP}}$
the collection (i.e., the set) of sets $E\subseteq\mathbb R^{+}$ which are strongly porous on the right at $0$ and completely strongly porous on the right at $0$ respectively. The set $\textbf{\emph{CSP}}$ was introduced and studied in \cite{DB} with slightly different, but equivalent definition.



\begin{definition}\label{D1.5}
Let $E\subseteq\mathbb R^{+}$ and $q>1.$ The $q$-blow up of $E$ is the set $$E(q):=\bigcup_{x\in E}(q^{-1}x, qx).$$
\end{definition}

The goal of the paper is to find some blow up characterizations for the intersection of maximal ideals $\textbf{\emph{I}}\subseteq\textbf{\emph{SP}}$ and for the ideal generated by $\textbf{\emph{CSP}}.$

\section{Down sets and ideals}
\hspace*{\parindent}Let $\textbf{\emph{A}}$ be a collection of sets. We shall say that $\textbf{\emph{A}}$ is a \emph{down set} if the implication

\begin{equation}\label{E1.5}
\left(B\in \textbf{\emph{A}}\,\, \wedge \,\,C\subseteq B\right)\Rightarrow \left(C\in \textbf{\emph{A}}\right)
\end{equation}
holds for all sets $C$ and $B$ (cf. \cite[p. 29]{Har}). If $\mbox{\boldmath$\varGamma$}$ is an arbitrary collection of sets, we write $$V=V(\mbox{\boldmath$\varGamma$}):=\bigcup_{ A\in \mbox{\scriptsize\boldmath$\varGamma$}}A.$$

\medskip


\begin{definition}
A collection $\textbf{I}$ of subsets of a set $X$ is an ideal on $X$ if the following conditions hold:
\newline $\emph{(i)}$ $\textbf{I}$ is a down set;
\newline $\emph{(ii)}$ $B\cup C\in \textbf{I}$ for all $B, C\in \textbf{I};$
\newline $\emph{(iii)}$ $X\not \in\textbf{I}.$
\end{definition}

Let $\mbox{\boldmath$\varGamma$}$ be a down set.
Define a set $\emph{I}(\mbox{\boldmath$\varGamma$})\subseteq 2^{V}$ by the rule

\begin{equation}\label{E1.7}
(B\in\emph{I}(\mbox{\boldmath$\varGamma$}))\Leftrightarrow(\exists \,\,n\in\mathbb N\,\,\exists\,\, A_1, ..., A_n\in\mbox{\boldmath$\varGamma$}: B=\bigcup_{j=1}^{n}A_j).
\end{equation}
If $V\not\in\emph{I}(\mbox{\boldmath$\varGamma$}),$ then $\emph{I}(\mbox{\boldmath$\varGamma$})$ is an ideal on $V$ such that $\mbox{\boldmath$\varGamma$}\subseteq\emph{I}(\mbox{\boldmath$\varGamma$})$ and the implication $$(\mbox{\boldmath$\varGamma$}\subseteq\mbox{\boldmath$\mathfrak{I}$})\Rightarrow (\emph{I}(\mbox{\boldmath$\varGamma$})\subseteq\mbox{\boldmath$\mathfrak{I}$})$$ holds for every ideal $\mbox{\boldmath$\mathfrak{I}$}$ on $V.$ In what follows we shall say that $\emph{I}(\mbox{\boldmath$\varGamma$})$ is the \emph{ideal generated by $\mbox{\boldmath$\varGamma$}.$}

\begin{definition}\label{D1.4}
Let $\mbox{\boldmath$\varGamma$}$ be an arbitrary collection of sets. An ideal $\textbf{I}$ on $V=V(\mbox{\boldmath$\varGamma$})$ is $\mbox{\boldmath$\varGamma$}$-maximal if $\textbf{I}\subseteq\mbox{\boldmath$\varGamma$}$ and the implication
\begin{equation}\label{E1.8}
(\textbf{I}\subseteq\mbox{\boldmath$\mathfrak{I}$}\subseteq\mbox{\boldmath$\varGamma$})\Rightarrow(\textbf{I}=\mbox{\boldmath$\mathfrak{I}$})
\end{equation} holds for every ideal $\mbox{\boldmath$\mathfrak{I}$}$ on $V.$
\end{definition}

Write $M(\mbox{\boldmath$\varGamma$})$ for the set of $\mbox{\boldmath$\varGamma$}$-maximal ideals and define an ideal $\hat{I} (\mbox{\boldmath$\varGamma$})$ as
\begin{equation}\label{E1.9}
\hat{I} (\mbox{\boldmath$\varGamma$}):=\bigcap_{\textbf{\emph{I}}\in M\left(\mbox{\scriptsize\boldmath$\varGamma$}\right)}\textbf{\emph{I}},
\end{equation}
i. e., $\hat{I} (\mbox{\boldmath$\varGamma$})$ is the intersection of $\mbox{\boldmath$\varGamma$}$-maximal ideals.

\medskip

The paper contains the following main results.

\medskip

\noindent $\bullet$ \emph{A characteristic property of sets which belong to the intersection $\hat{I} (\mbox{\boldmath$\varGamma$})$ of $\mbox{\boldmath$\varGamma$}$-maximal ideals with an arbitrary down set $\mbox{\boldmath$\varGamma$}$.} (See Theorem~\ref{L2.2}).

\medskip

\noindent $\bullet$ \emph{The blow up characterizations of the ideals $\hat I(\textbf{SP})$ and $I(\textbf{CSP}).$} (See theorems \ref{Th2.20} and \ref{Th2.19}).

\medskip

\noindent $\bullet$ \emph{The proper inclusion
$
I(\textbf{CSP})\subset\hat{I}(\textbf{SP}).
$} (See Corollary~\ref{Col(result)} and Example~\ref{ex(result)}).


\begin{remark}\label{Rem1}
The sets $\textbf{\emph{SP}}$ and  $\textbf{\emph{CSP}}$ are down sets and no one from these sets is an ideal on $\mathbb R^{+}.$ The definition of ideals via the down sets seems to be new, but it is easy to prove that this definition is equivalent to the usual definition of the ideals of sets. (See, for example, \cite[p. 202]{HI}).
\end{remark}
\begin{remark}\label{Rem1*}
The $\mbox{\boldmath$\varGamma$}$-maximal ideals are a generalization of the prime ideals. Indeed, if $\mbox{\boldmath$\varGamma$}=2^{V}$ and $\textbf{\emph{I}}$ is an ideal on $V,$ then it may be proved that $\textbf{\emph{I}}$ is a prime ideal on $V$ if and only if $\textbf{\emph{I}}$ is $\mbox{\boldmath$\varGamma$}$-maximal.
\end{remark}

\section {A property of the intersection of $\mbox{\boldmath$\varGamma$}$-maximal ideals}
\hspace*{\parindent}

We start with a useful property of an arbitrary $\mbox{\boldmath$\varGamma$}$-maximal ideal.

\begin{lemma}\label{L2.1}
Let $\mbox{\boldmath$\varGamma$}$ be a nonempty collection of sets. The following two statements are equivalent:
\begin{enumerate}
\item[\rm(i)]\textit{$\mbox{\boldmath$\varGamma$}$ is a down set and $V(\mbox{\boldmath$\varGamma$})\not\in\mbox{\boldmath$\varGamma$};$}

\item[\rm(ii)]\textit{For every $A\in\mbox{\boldmath$\varGamma$}$ there exists a $\mbox{\boldmath$\varGamma$}$-maximal ideal $\textbf{I}$ such that $A\in\textbf{I}.$}
\end{enumerate}
\end{lemma}
\begin{proof}
$\textrm{(ii)}\Rightarrow \textrm{(i)}$. Assume that $\textrm{(ii)}$ holds. Let $A\in\mbox{\boldmath$\varGamma$}.$ Using $\textrm{(ii)},$ we can find a $\mbox{\boldmath$\varGamma$}$-maximal ideal $\textbf{\emph{I}}\ni A$. Then $2^{A}\subseteq\textbf{\emph{I}}\subseteq\mbox{\boldmath$\varGamma$}$ holds. Hence $\mbox{\boldmath$\varGamma$}$ is a down set. Suppose now that $V\in\mbox{\boldmath$\varGamma$}.$ By (ii), there is a $\mbox{\boldmath$\varGamma$}$-maximal ideal $\textbf{\emph{I}}$ such that
\begin{equation}\label{E2.0}
V\in\textbf{\emph{I}}.
\end{equation}
The ideal $\textbf{\emph{I}}$ is an ideal on $V.$ Hence $V\not\in\textbf{\emph{I}},$ contrary to \eqref{E2.0}.

$\textrm{(i)}\Rightarrow \textrm{(ii)}$. Suppose that (i) holds. Let $A\in\mbox{\boldmath$\varGamma$}.$ Then $2^{A}\subseteq\mbox{\boldmath$\varGamma$}$ and $2^{A}$ is an ideal on $V.$ Using Zorn's Lemma, we can find a $\mbox{\boldmath$\varGamma$}$-maximal ideal  $\textbf{\emph{I}}$ such that $\textbf{\emph{I}}\supseteq 2^{A}.$ It is clear that $A\in\textbf{\emph{I}}$ holds. The implication $\textrm{(i)}\Rightarrow \textrm{(ii)}$ follows.
\end{proof}

Let $\mbox{\boldmath$\varGamma$}$ be a collection of sets. We shall denote by $I^{*}(\mbox{\boldmath$\varGamma$})$ the collection of sets $S$ satisfying the condition
\begin{equation}\label{E2.1}
S\cup B\in\mbox{\boldmath$\varGamma$}
\end{equation} for every $B\in\mbox{\boldmath$\varGamma$}.$
\begin{remark}\label{rem4.1*}
It is clear that $I^{*}(\mbox{\boldmath$\varGamma$})$ is a down set, if $\mbox{\boldmath$\varGamma$}$ is a down set.
\end{remark}
\begin{lemma}\label{L(new)}
If $\mbox{\boldmath$\varGamma$}$ is a nonempty collection of sets, then
\begin{equation*}
(V(\mbox{\boldmath$\varGamma$})\in\mbox{\boldmath$\varGamma$})\Leftrightarrow (V(\mbox{\boldmath$\varGamma$})\in I^{*}(\mbox{\boldmath$\varGamma$}))
\end{equation*}
holds.
\end{lemma}

\begin{proof}
Let $V\in\mbox{\boldmath$\varGamma$}.$ Then we have
$
B\cup V=V\in\mbox{\boldmath$\varGamma$}
$
for every $B\in\mbox{\boldmath$\varGamma$}.$ Hence $V\in I^{*}(\mbox{\boldmath$\varGamma$}).$

$\quad\,$ Let now $V\in I^{*}(\mbox{\boldmath$\varGamma$})$ and $B\in\mbox{\boldmath$\varGamma$}.$ The inclusion $B\subseteq V$ holds. Thus, $V=B\cup V\in\mbox{\boldmath$\varGamma$}.$
\end{proof}

\begin{theorem}\label{L2.2}
Let $\mbox{\boldmath$\varGamma$}$ be a nonempty down set and let
\begin{equation}\label{E2.2}
V(\mbox{\boldmath$\varGamma$})\not\in \mbox{\boldmath$\varGamma$}.
\end{equation} Then the equality
\begin{equation}\label{E2.3}
I^{*}(\mbox{\boldmath$\varGamma$})=\hat{I}(\mbox{\boldmath$\varGamma$})
\end{equation} holds where $\hat{I}(\mbox{\boldmath$\varGamma$})$ is defined by \eqref{E1.9}.
\end{theorem}

\begin{proof}
Let us prove the inclusion
\begin{equation}\label{E2.4}
I^{*}(\mbox{\boldmath$\varGamma$})\subseteq\hat{I}(\mbox{\boldmath$\varGamma$}).
\end{equation}
Using \eqref{E1.9}, we can see that \eqref{E2.4} holds if and only if
\begin{equation}\label{E2.5}
A\in\textbf{\emph{I}} \quad \mbox{for every $\mbox{\boldmath$\varGamma$}$-maximal ideal $\textbf{\emph{I}}$ and every $A\in I^{*}(\mbox{\boldmath$\varGamma$}).$}
\end{equation}
Let $A$ be an arbitrary element of $I^{*}(\mbox{\boldmath$\varGamma$})$ and let $\textbf{\emph{I}}$ be a $\mbox{\boldmath$\varGamma$}$-maximal ideal. Define a set $\textbf{\emph{I}}(A)$ as
\begin{equation}\label{E2.6}
\textbf{\emph{I}}(A):=\{B\cup K: B\subseteq A \,\,\mbox{and}\,\, K\in\textbf{\emph{I}}\}.
\end{equation}
The trivial inclusion $\varnothing\subseteq A$ implies that $\textbf{\emph{I}}\subseteq\textbf{\emph{I}}(A).$
It follows from Definition~\ref{D1.4} that $\textbf{\emph{I}}\subseteq\mbox{\boldmath$\varGamma$}$. Since $I^{*}(\mbox{\boldmath$\varGamma$})$ is a down set (see Remark~\ref{rem4.1*}), the relations $$B\subseteq A\in I^{*}(\mbox{\boldmath$\varGamma$})\quad\mbox{and}\quad K\in\textbf{\emph{I}}\subseteq\mbox{\boldmath$\varGamma$}$$ yield
\begin{equation}\label{eqvot}
B\cup K\in\mbox{\boldmath$\varGamma$}.
\end{equation}
Hence \begin{equation}\label{eqvot*}\textbf{\emph{I}}(A)\subseteq\mbox{\boldmath$\varGamma$}.\end{equation} Moreover, \eqref{eqvot}, \eqref{E2.6} and \eqref{E2.2} imply that $V\not\in\textbf{\emph{I}}(A).$
Since $\textbf{\emph{I}}$ and $\mbox{\boldmath$\varGamma$}$ are down sets, the definition of $I^{*}(\mbox{\boldmath$\varGamma$})$ and \eqref{E2.6} imply that $\textbf{\emph{I}}(A)$ is a down set. If, for $i=1,2,$ $B_i \cup K_i\in\textbf{\emph{I}}(A)$ with $B_i \subseteq A$ and $K_i\in\textbf{\emph{I}},$  then, by the definition of ideals, $K_1 \cup K_2\in\textbf{\emph{I}}$ and, moreover, $B_1 \cup B_2 \subseteq A.$ Consequently, from the equality
\begin{equation*}
(B_1 \cup K_1)\cup(B_2 \cup K_2)=(B_1 \cup B_2)\cup(K_1 \cup K_2)
\end{equation*} we obtain
\begin{equation*}
(B_1 \cup K_1)\cup(B_2 \cup K_2)\in\textbf{\emph{I}}(A).
\end{equation*}
Hence $\textbf{\emph{I}}(A)$ is an ideal on $V.$ Since $\textbf{\emph{I}}\subseteq\textbf{\emph{I}}(A)$ and $\textbf{\emph{I}}$ is $\mbox{\boldmath$\varGamma$}$-maximal, from $\eqref{eqvot*}$ and $\eqref{E1.8}$ we obtain the equality
\begin{equation}\label{E2.7}
\textbf{\emph{I}}(A)=\textbf{\emph{I}}.
\end{equation}
The membership $A\in\textbf{\emph{I}}(A)$ and \eqref{E2.7} yield \eqref{E2.5}.

Consider now the inclusion
\begin{equation}\label{E2.8}
\hat{I}(\mbox{\boldmath$\varGamma$})\subseteq I^{*}(\mbox{\boldmath$\varGamma$}).
\end{equation}
If \eqref{E2.8} does not hold, then we can find $A\in\hat{I}(\mbox{\boldmath$\varGamma$})$ and $B\in\mbox{\boldmath$\varGamma$}$ so that
\begin{equation}\label{E2.9}
A\cup B\not\in\mbox{\boldmath$\varGamma$}.
\end{equation} By Lemma~\ref{L2.1}, there is a $\mbox{\boldmath$\varGamma$}$-maximal ideal $\textbf{\emph{I}}$ such that $B\in\textbf{\emph{I}}.$ The membership $A\in\hat{I}(\mbox{\boldmath$\varGamma$})$ yields that $A\in \textbf{\emph{I}}.$
Since $\textbf{\emph{I}}$ is an ideal, from $A\in\textbf{\emph{I}}$ and $B\in\textbf{\emph{I}}$ it follows that $A\cup B\in\textbf{\emph{I}}\subseteq\mbox{\boldmath$\varGamma$},$ contrary to \eqref{E2.9}.
\end{proof}


\begin{corollary}\label{C2.1}
Let $\mbox{\boldmath$\varGamma$}$ be a nonempty down set. Then the collection $I^{*}(\mbox{\boldmath$\varGamma$})$ is an ideal on $V$ if and only if $V\not\in\mbox{\boldmath$\varGamma$}.$
\end{corollary}
\begin{proof}
The intersection of an arbitrary nonempty set of ideals is an ideal. The set of $\mbox{\boldmath$\varGamma$}$-maximal ideals is nonempty, because $\mbox{\boldmath$\varGamma$}\ne\varnothing.$ Consequently, $\hat I(\mbox{\boldmath$\varGamma$})$ is an ideal on $V=V(\mbox{\boldmath$\varGamma$}).$ Hence, by Theorem~\ref{L2.2}, $I^{*}(\mbox{\boldmath$\varGamma$})$ is an ideal on $V.$

Conversely, if $I^{*}(\mbox{\boldmath$\varGamma$})$ is an ideal on $V,$ then condition (iii) from the definition of ideals implies that $V\not\in I^{*}(\mbox{\boldmath$\varGamma$}).$ Using Lemma~\ref{L(new)}, we obtain that $V\not\in\mbox{\boldmath$\varGamma$}.$
\end{proof}

\begin{remark}
If $\mbox{\boldmath$\varGamma$}$ is a down set and $V(\mbox{\boldmath$\varGamma$})\in\mbox{\boldmath$\varGamma$},$ then, as is easily seen, the equality $\hat I(\mbox{\boldmath$\varGamma$})=\{\varnothing\}$ holds, so that, in this case, the question about the structure of $\hat I(\mbox{\boldmath$\varGamma$})$ is trivial.
\end{remark}



\section{Blow up of sets}\hspace*{\parindent}
Recall that for $q>1$ and $E\subseteq\mathbb R^{+}$ we define the $q$-blow up of $E$ as
\begin{equation}\label{eq4.12}
E(q):=\bigcup_{x\in E}(q^{-1}x, qx).
\end{equation}

\begin{remark}\label{rem4.5}
For all $E\subseteq\mathbb R^{+}$ and $q>1,$ we have
\begin{equation}\label{eq4.13}
(0\not \in E)\Leftrightarrow (E(q)\supseteq E).
\end{equation}
Indeed, the implication $(0\not \in E)\Rightarrow (E(q)\supseteq E)$ is evident. Conversely, suppose that $0\in E.$ Since $0\not \in (q^{-1}x, qx)$ for every nonzero $x$ and $(q^{-1}0, q0)=(0,0)=\varnothing,$ we obtain $0\not \in E(q).$ Thus \eqref{eq4.13} follows.
\end{remark}

\begin{lemma}\label{L13}
Let $0<a<b<\infty.$ The following statements hold.
\begin{enumerate}
\item[\rm(i)]\textit{If $q\ge\frac{b}{a}$ and $\varnothing\ne E\subseteq (a,b),$ then the set $E(q)$ is an open interval such that $E(q)\supseteq (a,b).$}

\item[\rm(ii)]\textit{If $E=(a,b),$ then $E(q)=(q^{-1}a, qb)$ for every $q>1.$}

\end{enumerate}
\end{lemma}

The proof is simple and omitted here.

\begin{lemma}\label{L(new1)}
Let $A$ and $B$ be subsets of $\mathbb R^{+},$ let $t>0$ and let
\begin{equation}\label{eqv0}
(0, t)\cap B\subseteq(0, t)\cap A
\end{equation}
hold. Then the inclusion
\begin{equation}\label{eqv0*}
(0, tq^{-1})\cap B(q)\subseteq (0, tq^{-1})\cap A(q)
\end{equation}
holds for every $q>1.$
\end{lemma}
\begin{proof}
Let $q>1$ and let $x\in (0, tq^{-1})\cap B(q).$ Then we have
\begin{equation}\label{eqv1}
0<x<tq^{-1}
\end{equation}
and there is $y\in B$ such that
\begin{equation}\label{eqv2}
q^{-1}y<x<qy.
\end{equation}
It follows from \eqref{eqv1} and \eqref{eqv2}, that $q^{-1}y<x<tq^{-1}.$ Consequently, $y< t$ holds. The last inequality, $y\in B$ and \eqref{eqv0} imply $$y\in (0, t)\cap B\subseteq (0, t)\cap A,$$ so that $y\in (0, t)$ and $y\in A.$ These relations yield $$(q^{-1}y, qy)\subseteq (0, tq)\quad \mbox{and}\quad (q^{-1}y, qy)\subseteq A(q)$$ because $(E\subseteq F)\Rightarrow (E(q)\subseteq F(q))$ holds for all $E\subseteq\mathbb R^{+}, \, F\subseteq\mathbb R^{+}$ and $q>1.$ Consequently we have $$x\in(q^{-1}y, qy)\subseteq (0, tq)\cap A(q)$$ for every $x\in (0, tq^{-1})\cap B(q),$ i. e.
\begin{equation}\label{eqv3}
(0, tq^{-1})\cap B(q)\subseteq (0, tq)\cap A(q).
\end{equation}
The inclusion $(0, tq^{-1})\subseteq (0, tq)$ and \eqref{eqv3} imply that $$(0, tq^{-1})\cap B(q)\subseteq (0, tq^{-1})\cap (0, tq)\cap A(q)\subseteq (0, tq^{-1})\cap A(q).$$
Hence, inclusion \eqref{eqv0*} follows immediately.
\end{proof}
\begin{remark}
The constant $q^{-1}$ is exact in the sense that for every $c>q^{-1}$ we can find $A$ and $B$ so that \eqref{eqv0} holds and $$(0, ct)\cap B(q)\nsubseteq (0, ct)\cap A(q).$$ To see this, put $B=[t, \infty)$ and $A=\{0\}.$ Then we obtain $(0, t)\cap B(q)=(0, t)\cap A=(0, ct)\cap A(q)=\varnothing$ and $(0, ct)\cap B(q)=(0, ct)\cap [q^{-1}t, \infty)=[q^{-1}t, ct)\ne\varnothing$ for all $q>1$ and $c\in (q^{-1}, \infty).$
\end{remark}

\begin{lemma}\label{L13*}
Let $E\subseteq\mathbb R^{+}$ and $E\not\in \textbf{SP}.$ Then there are $q>1$ and $t>0$ such that the equality
\begin{equation}\label{w1}
E(q)\cap (0, t)=(0, t)
\end{equation}
holds.
\end{lemma}
\begin{proof}
Since $E$ is not strongly porous on the right at $0,$ there is $s\in (0, 1)$ such that
\begin{equation*}
\limsup_{h\to 0+}\frac{\lambda (E, 0, h)}{h}<s,
\end{equation*}
where $\lambda (E, 0, h)$ is the length of the largest open subinterval of $(0, h)$ that contains no point of $E$ (see Definition~\ref{D1.1}). Consequently, there exists $t>0$ such that, for every $y\in (0, t)\setminus E,$ there exists $x\in E$ satisfying the inequalities
\begin{equation*}
x<y \quad\mbox{and}\quad \frac{y-x}{y}<s.
\end{equation*}
These inequalities imply that
\begin{equation*}
x<y < \frac{x}{1-s}.
\end{equation*}
Hence, $y\in (q^{-1}x, qx)$ holds with $q=\frac{1}{1-s}.$ Thus, the inclusion
\begin{equation*}
(0, t)\setminus E\subseteq E(q)
\end{equation*}
holds for such $q.$ Since $E\cap (0, t)\subseteq E(q)$ holds for all $t>0$ and $q>1,$ we obtain
\begin{equation*}
(0, t)=(E\cap (0, t))\cup((0, t)\setminus E)\subseteq E(q)\cup E(q)= E(q).
\end{equation*}
Thus,
\begin{equation*}
(0, t)\subseteq (0,t)\cap E(q)\subseteq (0, t),
\end{equation*}
which implies \eqref{w1}.
\end{proof}

\section{Blow up of strongly porous at $0$ sets}

\hspace*{\parindent}Let us prove that the $q$-blow up preserves $\emph{\textbf{SP}}.$


\begin{lemma}\label{L14}
Let $E\subseteq\mathbb R^{+}$ and $q>1.$ Then $E$ belongs to $\textbf{SP}$ if and only if $E(q)$ belongs to $\textbf{SP}.$
\end{lemma}
\begin{proof}
Since $E(q)=(E\setminus\{0\})(q)$ and $$(E\in \textbf{\emph{SP}})\Leftrightarrow (E\setminus\{0\}\in \textbf{\emph{SP}}),$$ we may assume that
$
0\not\in E.
$ In accordance with \eqref{eq4.13}, this
assumption implies the inclusion
\begin{equation}\label{eq4.15}
E\subseteq E(q).
\end{equation}
Since $\textbf{\emph{SP}}$ is a down set, the implication $(E(q)\in \textbf{\emph{SP}})\Rightarrow (E\in \textbf{\emph{SP}})$ follows.

Let $E\in \textbf{\emph{SP}}.$ Then there is a sequence $\{(a_n, b_n)\}_{n\in\mathbb N}$ such that $0<a_n <b_n,$ $b_n\downarrow 0,$ $(a_n, b_n)\cap E=\varnothing$ and $\mathop{\lim}\limits_{n\to\infty}\frac{a_n}{b_n}=0.$ It is easy to prove that $qa_n <q^{-1}b_n$ and $(qa_n, q^{-1}b_n)\cap E(q)=\varnothing$ for all sufficiently large $n.$ Since
\begin{equation*}
\lim_{n\to\infty}\frac{qa_n}{q^{-1}b_n}=\lim_{n\to\infty}q^{2}\frac{a_n}{b_n}=0,
\end{equation*}
the set $E(q)$ is strongly porous on the right at $0.$ The implication $(E\in \textbf{\emph{SP}})\Rightarrow (E(q)\in \textbf{\emph{SP}})$ follows. Thus, $$(E\in \textbf{\emph{SP}})\Leftrightarrow (E(q)\in \textbf{\emph{SP}})$$ holds.
\end{proof}

\begin{corollary}\label{col}
Let $E\subseteq\mathbb R^{+}$ and $q>1.$ Then $E\in I^{*}(\textbf{SP})$ holds if and only if $E(q)\in I^{*}(\textbf{SP}).$
\end{corollary}
\begin{proof}
As in the proof of Lemma~\ref{L14}, we may suppose that $E(q)\supseteq E.$ This yields $(E(q)\in I^{*}(\textbf{\emph{SP}}))\Rightarrow (E\in I^{*}(\textbf{\emph{SP}})).$ Let $E\in I^{*}(\textbf{\emph{SP}}).$ The relation $E(q)\in I^{*}(\textbf{\emph{SP}})$ holds if and only if
\begin{equation}\label{eq4.18}
E(q)\cup B\in \textbf{\emph{SP}} \quad\mbox{for every $B\in\textbf{\emph{SP}}.$}
\end{equation}
Using the relation
\begin{equation*}
(B\in \textbf{\emph{SP}})\Leftrightarrow (B\setminus\{0\}\in \textbf{\emph{SP}})
\end{equation*}
we may consider only the case where $0\not\in B.$ The membership $E\in I^{*}(\textbf{\emph{SP}})$ implies $E\cup B\in\textbf{\emph{SP}}.$
Consequently, by statement (ii) of Lemma~\ref{L14}, we obtain
\begin{equation}\label{eq4.19}
E(q)\cup B(q)\in \textbf{\emph{SP}}.
\end{equation}
Since $0\not\in B,$ the inclusion $B\subseteq B(q)$ holds. The last inclusion and \eqref{eq4.19} yield \eqref{eq4.18}.
\end{proof}

Let $A$ and $B$ be nonempty subsets of $\mathbb R^{+}.$ We define $A\prec B$ if $b< a$ holds for every $b\in B$ and $a\in A.$ Furthermore, we set
$$A\preceq B \quad\mbox{if}\quad A = B \quad \mbox{or} \quad A\prec B.$$ The relation $\preceq$ is a partial order on the set of nonempty subsets of $\mathbb R^{+}.$
A chain (i.e., a linearly ordered set) $(P, \le_{P})$ is said to be well-ordered if every nonempty subset $X$ of $P$ contains a smallest element, i.e., an element $x\in X$ such that $x\le_{P} y$ for every $y\in X.$

It is easy to prove, that for every nonempty $A\subseteq\mathbb R^{+},$ the set $\mathrm{Cc}A$ of connected components of $A$ is a chain w. r. t. the partial order $\preceq$. Define a set $\mathrm{Cc}^{1} A$ by the rule: $$B \in\mathrm{Cc}^{1}A \quad\mbox{if}\quad B\in \mathrm{Cc}A\quad\mbox{and}\quad B\subset (0, 1].$$

\begin{lemma}\label{L15}
Let $\varnothing \ne E \subseteq \mathbb R^{+}$ and let $q>1.$
Then the chain $(\mathrm{Cc}^{1} E(q), \preceq)$ is well-ordered.
\end{lemma}

\begin{proof}
If there is $X\subseteq \mathrm{Cc}^{1}E(q),$ which does not have a smallest element, then there is a sequence $\{(a_i, b_i)\}_{i\in\mathbb N}$ such that $$(a_1, b_1)\succ (a_2, b_2)\succ ... (a_i, b_i)\succ(a_{i+1}, b_{i+1})\succ ...$$ with $(a_i, b_i)\in X$ for every $i\in\mathbb N.$ The equalities
\begin{equation*}
\ln a_{1}^{-1}=(\ln a_{1}^{-1}-\ln b_{1}^{-1})+\ln b_{1}^{-1}
\end{equation*}
\begin{equation*}
=(\ln a_{1}^{-1}-\ln b_{1}^{-1})+(\ln b_{1}^{-1}-\ln a_{2}^{-1})+(\ln a_{2}^{-1}-\ln b_{2}^{-1})+\ln b_{2}^{-1}
\end{equation*}
\begin{equation*}
= ... =\sum_{k=1}^{i+1}(\ln a_{k}^{-1}-\ln b_{k}^{-1})+\sum_{k=1}^{i}(\ln b_{k}^{-1}-\ln a_{k+1}^{-1})+\ln b_{i+1}^{-1}
\end{equation*}
and the inequalities
\begin{equation*}
\ln a_{k}^{-1}>\ln b_{k}^{-1}\ge \ln a_{k+1}^{-1}> \ln b_{k+1}^{-1}\ge 0,
\end{equation*}
$k=1, ..., i+1$ imply that
\begin{equation}\label{s1}
\ln a_{1}^{-1}\ge\sum_{k=1}^{i+1}(\ln a_{k}^{-1}-\ln b_{k}^{-1}).
\end{equation}
Since $X\subseteq \mathrm{Cc}^{1}E(q),$ the intersection $(a_k, b_k)\cap E$ is nonempty for every $k=1, ..., i.$ It follows directly from the definition of $q$-blow up, that the inclusion
\begin{equation}\label{s2}
(q^{-1}x, qx)\subseteq (a_k, b_k)
\end{equation}
holds for every $x\in E\cap (a_k, b_k).$ Conditions \eqref{s1} and \eqref{s2} yield the inequalities
\begin{equation*}
\ln a_{1}^{-1}\ge\sum_{k=1}^{i+1}\ln\frac{b_k}{a_k}\ge\sum_{k=1}^{i+1}\ln q^{2}=2(i+1)\ln q.
\end{equation*}
Letting $i\to\infty,$ we obtain the equality $\ln a_{1}^{-1}=\infty,$ contrary to $(a_1, b_1)\in \mathrm{Cc}^{1}E(q).$
\end{proof}

The proof of Lemma~\ref{L15} shows, in particular, that for given $q>1$ and $(a,b)\in \mathrm{Cc}^{1}E(q),$ the set $\{(c,d)\in \mathrm{Cc}^{1}E(q): (c,d)\preceq (a,b)\}$ is finite. This finiteness together with Lemma~\ref{L15} implies the following

\begin{corollary}\label{col4.8}
Let $\varnothing\ne E\subseteq\mathbb R^{+}$ and let $q>1.$ If $\mathrm{Cc}^{1}E(q)\ne\varnothing,$ then the chain $(\mathrm{Cc}^{1}E(q), \preceq)$ is isomorphic either the first infinite ordinal number $\omega$ or an initial segment of $\omega.$
\end{corollary}

For a set $E\subseteq\mathbb R^{+},$ we use the symbol $\mathrm{ac} E$ to denote the set of its accumulation points.

\begin{remark}\label{Rem3}
Let $E\subseteq\mathbb R^{+}$ and $q>1.$ Then $(\mathrm{Cc}^{1}E(q), \preceq)$ is isomorphic to $\omega$ if and only if $0\in \mathrm{ac} E(q)$ and $0\in \mathrm{ac} (\mathbb R^{+}\setminus E(q)).$ In particular, if $E\in \textbf{\emph{SP}},$ then $\mathrm{Cc}^{1}E(q)$ is isomorphic to $\omega$ if and only if $0\in \mathrm{ac} E.$
\end{remark}

Corollary~\ref{col4.8} means, in particular, that for every infinite $\mathrm{Cc}^{1}E(q)$ there is a unique sequence $\{(a_i, b_i)\}_{i\in\mathbb N}$ such that the logical equivalence
\begin{equation}\label{z1}((a, b)\in \mathrm{Cc}^{1} E(q))\Leftrightarrow (\exists \, i\in\mathbb N: (a,b)=(a_i, b_i))\end{equation} holds for every interval $(a, b)\subseteq\mathbb R^{+}$ and logical equivalence
\begin{equation}\label{z2}
((a_i, b_i)\prec (a_j, b_j))\Leftrightarrow (i<j)
\end{equation}
holds for all $i, j\in\mathbb N.$
If a sequence $\{(a_i, b_i)\}_{i\in\mathbb N}$ satisfies \eqref{z1} $-$ \eqref{z2} we shall write $$\mathrm{Cc}^{1} E(q)=\{(a_i, b_i)\}_{i\in\mathbb N}.$$
The following theorem is a blow up characterization of the ideal $\hat{I}(\textbf{SP}).$
\begin{theorem}\label{Th2.20}
Let $E\subseteq\mathbb R^{+}$ and $0\in \mathrm{ac} E.$ Then the following conditions are equivalent.
\begin{enumerate}
\item[\rm(i)]\textit{$E\in \hat I(\textbf{SP})$.}

\item[\rm(ii)]\textit{For every $q>1,$ the chain $\mathrm{Cc}^{1}E(q)$ is infinite, $\mathrm{Cc}^{1}E(q)=\{(a_i, b_i)\}_{i\in\mathbb N},$ and the inequality
\begin{equation}\label{th2.20eq1}
\limsup_{i\to\infty}\frac{b_i}{a_i}<\infty
\end{equation} holds. }
\end{enumerate}
\end{theorem}
\begin{proof}
$\textrm{(i)}\Rightarrow \textrm{(ii)}.$ In accordance with Theorem~\ref{L2.2}, the equality $\hat I(\textbf{SP})=I^{*}(\textbf{SP})$ holds, so that $(E\in\hat I(\textbf{\emph{SP}}))\Leftrightarrow (E\in I^{*}(\textbf{\emph{SP}}))$. Suppose that $E\in I^{*}(\textbf{\emph{SP}})$ and $q>1.$ Then, by Corollary~\ref{col}, $E(q)\in I^{*}(\textbf{\emph{SP}})$ holds. Since $\textbf{\emph{SP}}$ is a down set, it follows directly from the definition of $I^{*}(\textbf{\emph{SP}})$ that $I^{*}(\textbf{\emph{SP}})\subseteq \textbf{\emph{SP}}.$ Consequently, the equality $\mathrm{Cc}^{1}E(q)=\{(a_i, b_i)\}_{i\in\mathbb N}$ holds. (See Remark~\ref{Rem3}). Suppose that
\begin{equation}\label{th2.20eq2}
\limsup_{i\to\infty}\frac{b_i}{a_i}=\infty.
\end{equation}
Let us consider the set
\begin{equation*}
B:=\mathbb R^{+}\setminus\left(\bigcup_{i\in\mathbb N}(a_i, b_i)\right).
\end{equation*}
Definition~\ref{D1.1} and \eqref{th2.20eq2} imply that $B\in \textbf{\emph{SP}}.$ Consequently, by the definition of $I^{*}(\textbf{\emph{SP}})$ we must have $B\cup E(q)\in\textbf{\emph{SP}}.$ It is clear from the definition of $B$ that
\begin{equation*}
(0, b_1)\subseteq B\cup E(q).
\end{equation*}
Hence the interval $(0, b_1)$ must be strongly porous on the right at $0,$ contrary to Definition~\ref{D1.1}. Hence (i) implies (ii).

$\textrm{(ii)}\Rightarrow \textrm{(i)}.$ Suppose now, that condition (ii) holds, but $E\not\in I^{*}(\textbf{\emph{SP}}).$ Then, there is $B\in \textbf{\emph{SP}}$ such that $B\cup E\not \in \textbf{\emph{SP}}.$ By Lemma~\ref{L13*}, we can find $q>1$ and $t>0$ such that the $q$-blow-up of $B\cup E$ is a superset of the interval $(0, t),$ i. e.
\begin{equation}\label{th2.20eq3}
B(q)\cup E(q)\supseteq (0, t).
\end{equation}
Lemma~\ref{L14} shows that $B(q)\in\textbf{\emph{SP}}.$ Consequently, there is a sequence $\{(a_{j}^{*}, b_{j}^{*})\}_{j\in\mathbb N}$ of open intervals $(a_{j}^{*}, b_{j}^{*})$ such that
\begin{equation}\label{th2.20eq4}
0<a_{j}^{*}<b_{j}^{*}<\infty, \, a_{j}^{*}\downarrow 0, \, (a_{j}^{*}, b_{j}^{*})\cap B(q)=\varnothing\quad\mbox{and}\quad\lim_{j\to\infty}\frac{b_{j}^{*}}{a_{j}^{*}}=\infty
\end{equation}
hold for every $j\in\mathbb N.$ Inclusion \eqref{th2.20eq3} and relations \eqref{th2.20eq4} imply that $(a_{j}^{*}, b_{j}^{*})\subseteq E(q)$ hold for all sufficiently large $j\in\mathbb N.$ Using condition (ii) of the present lemma, we can find a subsequence $\{(a_{i_k}, b_{i_k})\}_{k\in\mathbb N}$ of the sequence $\{(a_{i}, b_{i})\}_{i\in\mathbb N},$ where $\{(a_{i}, b_{i})\}_{i\in\mathbb N}=\mathrm{Cc}^{1}E(q)$ and a subsequence $\{(a_{j_k}^{*}, b_{j_k}^{*})\}_{k\in\mathbb N}$ of the sequence $\{(a_{j}^{*}, b_{j}^{*})\}_{j\in\mathbb N}$, so that
$
(a_{j_k}^{*}, b_{j_k}^{*})\subseteq(a_{i_k}, b_{i_k})
$
for every $k\in\mathbb N.$ Consequently, we obtain
\begin{equation*}
\limsup_{i\to\infty}\frac{b_i}{a_i}\ge\limsup_{k\to\infty}\frac{b_{i_k}}{a_{i_k}}\ge\limsup_{k\to\infty}\frac{b_{j_k}^{*}}{a_{j_k}^{*}}=\lim_{j\to\infty}\frac{b_{j}^{*}}{a_{j}^{*}}=\infty,
\end{equation*}
contrary to \eqref{th2.20eq1}.
\end{proof}

\section {Ideal generated by \textbf{\emph{CSP}}} \hspace*{\parindent}

The goal of the present section is to obtain the blow up characterization of the ideal $I(\textbf{\emph{CSP}}).$









\bigskip

The following lemma is a direct consequence of Theorem~36 and Theorem~42 from \cite{DB}.

\begin{lemma}\label{L2CSP}
Let $E\subseteq\mathbb R.$ Then
$E\in \textbf{CSP}$ if and only if there
are $q>1,\, t>0$ and a decreasing sequence $\{x_n\}_{n\in\mathbb N},$ such that $x_n >0$ for all $n\in\mathbb N,$ $\mathop{\lim}\limits_{n\to\infty}\frac{x_{n+1}}{x_n}=0$ and $$E\cap (0, t)\subseteq\left(\bigcup_{n\in\mathbb N}(q^{-1}x_n, qx_n)\right)\cap (0, t).$$
\end{lemma}

In this section, for every $n\in\mathbb N,$ we denote by $\textbf{\emph{n}}$ the set $\{1,2,...,n\}.$

\begin{lemma}\label{L14*}
Let $E\subseteq\mathbb R^{+}$ and $q>1.$ Then the following logical equivalence
\begin{equation*}
(E\in I (\textbf{CSP}))\Leftrightarrow (E(q)\in I (\textbf{CSP}))
\end{equation*}
holds.
\end{lemma}

\begin{proof}
As in the proof of Lemma~\ref{L14}, we may assume that $0\not\in E.$ In accordance with Remark~\ref{rem4.5}, this
assumption implies the inclusion
\begin{equation}\label{eq4.15}
E\subseteq E(q).
\end{equation}
Now the implication $$(E(q)\in I(\textbf{\emph{CSP}}))\Rightarrow (E \in I(\textbf{\emph{CSP}}))$$ follows from \eqref{eq4.15}, because $I(\textbf{\emph{CSP}})$ is a down set. To prove the converse implication suppose that $E\in I(\textbf{\emph{CSP}}).$ Then there are $B_1,..., B_n \in \textbf{\emph{CSP}},$ such that $E=B_1 \cup ...\cup B_n.$ The last equality implies that $E(q)=B_{1}(q)\cup... \cup B_{n}(q).$ Consequently $E(q)\in I(\textbf{\emph{CSP}})$ holds if
$B_{j}(q)\in \textbf{\emph{CSP}}$
for every $j\in \textbf{\emph{n}}.$ By Lemma~\ref{L2CSP}, for every $j\in \textbf{\emph{n}},$ we can find $q_j >1,$ $t_j>0,$ and a decreasing sequence $\{x_{k,j}\}_{k\in\mathbb N}$ of positive numbers such that $\mathop{\lim}\limits_{k\to\infty}\frac{x_{k+1,j}}{x_{k, j}}=0$ and
\begin{equation}\label{eq4.17}
(0,t_j)\cap B_{j}\subseteq (0,t_j)\cap \bigcup_{k\in\mathbb N}(q_{j}^{-1}x_{k,j}, q_{j}x_{k,j}).
\end{equation}
Statement (ii) of Lemma~\ref{L13}, Lemma~\ref{L(new1)} and \eqref{eq4.17} imply $$(0,t_{j}q^{-1})\cap B_{j}(q)\subseteq (0,t_{j}q^{-1})\cap \bigcup_{k\in\mathbb N}(q^{-1}q_{j}^{-1}x_{k,j}, qq_{j}x_{k,j}).$$ Hence, by Lemma~\ref{L2CSP}, the statement $B_{j}(q)\in \textbf{\emph{CSP}}$ holds for every $j\in \textbf{\emph{n}}.$
\end{proof}

\begin{lemma}\label{L16}
Let $E \subseteq \mathbb R^{+},$ $q>1$ and let $\mathrm{Cc}^{1} E(q)=\{(a_i, b_i)\}_{i\in\mathbb N}.$ Suppose that

\begin{equation}\label{eq4.24}
\limsup_{i\to\infty}\frac{b_i}{a_{i}}<\infty
\end{equation}
and there is $N\in\mathbb N$ such that

\begin{equation}\label{eq4.25}
\lim_{n\to\infty}\bigvee_{j=0}^{N}\frac{a_{n+j}}{b_{n+j+1}}=\infty
\end{equation}
where
\begin{equation*}
\bigvee_{j=0}^{N}\frac{a_{n+j}}{b_{n+j+1}}=\max\left\{ \frac{a_n}{b_{n+1}}, \frac{a_{n+1}}{b_{n+2}}, ..., \frac{a_{n+N}}{b_{n+N+1}} \right\}.
\end{equation*}
Then there are $B_1,..., B_{2N+2} \in \textbf{CSP}$ such that

\begin{equation}\label{eq4.26}
E\subseteq B_1 \cup ... \cup B_{2N+2}.
\end{equation}

\end{lemma}
\begin{proof}
Suppose $N\in\mathbb N$ is a number such that \eqref{eq4.25} holds. Let us define a sequence $\{F_k\}_{k\in\mathbb N}$ of sets $F_{k}\subseteq\mathbb N$ as $F_{1}:=\{1, ..., N+1\},$ $F_{2}:=\{(N+1)+1, ..., 2(N+1)\},$ $F_{3}:=\{2(N+1)+1, ..., 3(N+1)\}$ and so on. It is clear that $\mathop{\bigcup}\limits_{k=1}^{\infty}F_{k}=\mathbb N$ and $F_{k_1}\cap F_{k_2}=\varnothing$ if $k_1 \ne k_2,$ and
\begin{equation}\label{eq4.27}
|F_{k}|=N+1 \quad\mbox{for every $k\in\mathbb N.$}
\end{equation}
Let $m_k \in F_k$ be a number satisfying the condition

\begin{equation}\label{eq4.28}
\frac{a_{m_k}}{b_{m_{k}+1}}=\bigvee_{n\in F_k}\frac{a_n}{b_{n+1}}.
\end{equation}
It follows from \eqref{eq4.25}, \eqref{eq4.27} and \eqref{eq4.28} that

\begin{equation}\label{eq4.29}
\lim_{k\to\infty}\frac{a_{m_k}}{b_{m_{k}+1}}=\infty.
\end{equation}
The definition of $F_k$ and \eqref{eq4.27} imply the double inequality

\begin{equation}\label{eq4.30}
1\le m_{k+1}-m_{k}\le 2N+1.
\end{equation}
For every $k\in\mathbb N$ denote by $\mathfrak{F}_k$ the set of all connected components of $E(q),$ which lie between $[b_{m_{k}+2}, a_{m_{k}+1}]$ and $[b_{m_{k}+1}, a_{m_{k}}],$

\begin{equation}\label{eq4.31}
\mathfrak{F}_k:=\{(a_{n}, b_{n}): [b_{m_{k}+2}, a_{m_{k}+1}]\succ (a_{n}, b_{n})\succ [b_{m_{k}+1}, a_{m_{k}}]\}.
\end{equation}
It easy to show that
\begin{equation}\label{eq4.32}
\bigcup_{k=m_1}^{\infty}(a_{k+1}, b_{k+1})=\bigcup_{k=1}^{\infty}\mathfrak F_{k}
\end{equation}
and $\mathfrak F_{i}\cap \mathfrak F_{j}=\varnothing$ if $i\ne j.$
From \eqref{eq4.30} it also follows that $1\le |\mathfrak F_{k}|\le 2N+1$ for every $k\in\mathbb N.$
Consequently, for every $k\in\mathbb N,$ the elements of $\mathfrak F_{k}$ can be numbered (with some repetitions if it is necessary) in a finite sequence $(a_{k,1}, b_{k,1}), \, (a_{k,2}, b_{k,2}), ..., (a_{k,2N+1}, b_{k,2N+1}).$ Using the inclusion $$E(q)\subseteq \bigcup_{n=1}^{\infty}(a_{n+1}, b_{n+1})\cup (a_1, \infty)$$ and \eqref{eq4.32} we obtain

\begin{equation}\label{eq4.33}
E(q)\subseteq \bigcup_{k\in\mathbb N}\left(\bigcup_{j=1}^{2N+1}(a_{k,j}, b_{k,j})\right)\cup (a_{m_1}, \infty)
=\bigcup_{j=1}^{2N+1}\left(\bigcup_{k\in\mathbb N}(a_{k,j}, b_{k,j})\right)\cup (a_{m_1}, \infty).
\end{equation}
Write $$B_{j}:=\bigcup_{k\in\mathbb N}(a_{k,j}, b_{k,j})$$ for every $j\in \textbf{2\emph{N}+1},$ where $\textbf{2\emph{N}+1}=\{1, ..., 2N+1\},$ and put $B_{2N+2}:=\{0\}\cup (a_{m_1}, \infty).$ Now we have $E\subseteq E(q)\cup\{0\}\subseteq B_{1}\cup...\cup B_{2N+2}.$ It still remains to prove that $B_{j}\in\textbf{\emph{ CSP}}$ for $j=1,..., 2N+2.$  The statement $B_{2N+2}\in \textbf{\emph{CSP}}$ is clear. Let $j\in \textbf{2\emph{N}+1}.$ In accordance with the Definition~\ref{D1.2}, the statement $B_{j}\in \textbf{\emph{CSP}}$ holds if for every $\tilde h=\{h^{l}\}_{l\in\mathbb N}\in\tilde B_{j}$ there is $\tilde a=\{a^{l}\}_{l\in\mathbb N}\in\tilde H(B_j)$ such that $\tilde h\asymp \tilde a.$
Inequality \eqref{eq4.24} and the definition of $B_{j}$ imply that there is a positive constant $c>1$ such that $$a_{k,j}\le x\le ca_{k,j}$$ for every $x\in (a_{k,j}, b_{k,j})$ and every $k\in\mathbb N.$ Consequently, if $\{h^{l}\}_{l\in\mathbb N}\in\tilde B_{j},$ then we have $\{h^{l}\}_{l\in\mathbb N}\asymp\{a^{l}\}_{l\in\mathbb N},$ where, for every $l\in\mathbb N,$ $a^{l}$ is the left endpoint of the interval $(a_{k, j}, b_{k, j})$ which contains $h^l.$ Hence, $B_{j}\in \textbf{\emph{CSP}}$ holds if $\{a_{k, j}\}_{k\in\mathbb N}\in\tilde H (B_j),$ which is equivalent to
\begin{equation}\label{eq4.34}
\lim_{k\to\infty}\frac{a_{k,j}}{b_{k+1, j}}=\infty.
\end{equation}
Let us prove \eqref{eq4.34}.
It follows from \eqref{eq4.31} that $$[b_{m_{k}+2}, a_{m_{k}+1}]\succ (a_{k, j}, b_{k, j})\succ [b_{m_{k}+1}, a_{m_{k}}]$$ and $$[b_{m_{k}+3}, a_{m_{k}+2}]\succ (a_{k+1, j}, b_{k+1, j})\succ [b_{m_{k}+2}, a_{m_{k}+1}].$$
Hence we have
$$(a_{k+1, j}, b_{k+1, j})\succ [b_{m_{k}+2}, a_{m_{k}+1}]\succ (a_{k, j}, b_{k, j}) .$$
Consequently the inequality
$$\frac{a_{k,j}}{b_{k+1, j}}\le \frac{a_{m_{k+1}}}{b_{m_{k+2}}}$$
holds. The last inequality and \eqref{eq4.29} imply \eqref{eq4.34}.
\end{proof}
\begin{corollary}
Let $E\subseteq\mathbb R^{+}.$ If there are $N\in\mathbb N$ and $q>1$ so that $\mathrm{Cc}^{1} E(q)$ is infinite and conditions \eqref{eq4.24} and \eqref{eq4.25} hold, then $E\in I(\textbf{CSP}).$
\end{corollary}

In the following lemma, as in Lemma~\ref{L16}, the equality $\mathrm{Cc}^{1} E(q)=\{(a_i, b_i)\}_{i\in\mathbb N}$ means that conditions \eqref{z1} and \eqref{z2} are satisfied.

\begin{lemma}\label{L2.3}
Let $E\in I(\textbf{CSP})$ and let $0\in \mathrm{ac}E.$ Then $\mathrm{Cc}^{1} E(q)=\{(a_i, b_i)\}_{i\in\mathbb N}$ for every $q>1,$ and there are $q_{0}>1$ and $M\in\mathbb N$ such that the conditions
\begin{equation}\label{L2.3eq1}
\limsup_{i\to\infty}\frac{b_i}{a_i}<\infty
\end{equation}
and
\begin{equation}\label{L2.3eq2}
\lim_{n\to\infty}\bigvee_{j=0}^{M}\frac{a_{n+j}}{b_{n+j+1}}=\infty
\end{equation}
hold for every $q> q_0.$
\end{lemma}

\begin{proof}
It follows from the definition of $I(\emph{\textbf{CSP}})$ that there is $N\in\mathbb N$ such that
\begin{equation}\label{L2.3eq3}
E=B_1 \cup ... \cup B_N \quad\mbox{with some $B_1, ..., B_N\in\emph{\textbf{CSP}}.$}
\end{equation}
Let $\textbf{\emph{N}}=\{1, ...,
N\}.$ We may assume $0\in \mathrm{ac} B_{j}$ for every
$j\in\textbf{\emph{N}}.$ Indeed, if $0\not\in \mathrm{ac} B_{j}$ for all
$j\in\textbf{\emph{N}},$ then $$0\not\in \mathrm{ac} (B_1 \cup ... \cup
B_N)=\mathrm{ac} E,$$ contrary to the condition $0\in \mathrm{ac} E$. Hence, there is
$j_1 \in \textbf{\emph{N}}$ such that $0\in \mathrm{ac} B_{j_1}.$ Write
$$\textbf{\emph{J}}_{0}:=\{j\in \textbf{\emph{N}}: \mathrm{ac} B_j \not\ni
0\}, \quad\textbf{\emph{J}}_{1}:=\{j\in \textbf{\emph{N}}: \mathrm{ac} B_j \ni
0\}\quad \mbox{and} \quad  B_{j}':=
B_{j}\cup\left(\mathop{\bigcup}\limits_{i\in\textbf{\emph{J}}_{0}}B_i\right)$$
for every $j\in\textbf{\emph{J}}_{1}.$ Renumbering the elements of $\textbf{\emph{N}},$ we may also assume that $\textbf{\emph{J}}_{1}=\{1, ..., N_1\}$ with $N_{1}\le N.$ Then
the representation
$$E=B_{1}'\cup ...\cup B_{N_1}' $$ holds with $B_{j}'\in
\textbf{\emph{CSP}}$ and $\mathrm{ac} B_{j}'\ni 0$ for every $j\in
\textbf{\emph{N}}_{1}.$ Without loss of generality, we put $\textbf{\emph{N}}_{1}=\textbf{\emph{N}}$ and $B_{j}=B_{j}'$ for every $j\in \textbf{\emph{N}}_{1}.$

Using Lemma~\ref{L2CSP}, for every $j\in \textbf{\emph{N}},$ we can find
$q_{j}\in (1, \infty)$ and a strictly decreasing sequence $\{x_{j,
n}\}_{n\in\mathbb N}$ with
\begin{equation}\label{L2.3eq4}
\lim_{n\to\infty}\frac{x_{j, n+1}}{x_{j, n}}=0 \quad \mbox{and}
\quad \lim_{n\to\infty}x_{j, n}=0,
\end{equation} so that the inclusion
\begin{equation}\label{L2.3eq6}
B_{j}\cap (0, x_{j, 1})\subseteq\bigcup_{n\in\mathbb N}(q_{j}^{-1}x_{j, n},
q_{j}x_{j, n})
\end{equation}
holds.
Write
\begin{equation}\label{L2.3eq7}
B_{j, n}:=B_{j}\cap(q_{j}^{-1}x_{j, n}, q_{j}x_{j, n})
\end{equation}
for all $n\in\mathbb N$ and $j\in \textbf{\emph{N}},$ and define
\begin{equation}\label{L2.3eq8}
B_{j, 0}:=B_{j}\cap[q_{j}x_{j, 1}, \infty)
\end{equation}
for every $j\in \textbf{\emph{N}}.$ Inclusion \eqref{L2.3eq6} implies that
\begin{equation}\label{L2.3eq9}
B_{j}\setminus\{0\}=\bigcup_{n=0}^{\infty}B_{j, n}
\end{equation}
and from \eqref{L2.3eq3} it follows that
\begin{equation}\label{L2.3eq10}
E\setminus\{0\}=\bigcup_{j=1}^{N}\left(\bigcup_{n=0}^{\infty}B_{j,
n}\right).
\end{equation}
Replacing the sequences $\{x_{j, n}\}_{n\in\mathbb N}$ by suitable
subsequences, we may assume that
\begin{equation}\label{L2.3eq11}
B_{j, n}\ne\varnothing\quad\mbox{for every $j\in \textbf{\emph{N}}$ and $n\in\mathbb N.$}
\end{equation}
Recall that $0\in \mathrm{ac} B_{j}$ holds for every $j\in\textbf{\emph{N}}.$ Let
$q\ge\mathop{\bigvee}\limits_{j=1}^{N}q_{j}^{2}.$ The statement (i) of
Lemma~\ref{L13} and \eqref{L2.3eq11} imply that $B_{j, n}(q)$ are open
intervals. Write
\begin{equation}\label{L2.3eq12}
B_{j, n}(q):=(r_{j, n}, s_{j, n}), \quad n\in\mathbb N, \, j\in
\textbf{\emph{N}}.
\end{equation}
Consequently,
from statement (ii) of Lemma~\ref{L13} and $$B_{j, n}\subseteq (q_{j}^{-1}x_{j, n}, q_{j}x_{j, n}) \quad
\mbox{and} \quad q\ge\bigvee_{j=1}^{N}q_{j}^{2}$$ it follows that
$$(r_{j, n}, s_{j, n})=B_{j, n}(q)\subseteq (q^{-1}q_{j}^{-1}x_{j, n}, qq_{j}x_{j, n})\subseteq (q^{-\frac{3}{2}}x_{j, n}, q^{\frac{3}{2}}x_{j, n}).$$
Hence the inequality
\begin{equation}\label{L2.3eq13}
\frac{s_{j, n}}{r_{j, n}}\le q^{3}
\end{equation}
holds for all $n\in\mathbb N$ and $j\in \textbf{\emph{N}}.$
Since $$x_{j, n}\in (s_{j, n}, r_{j, n})\quad\mbox{and}\quad x_{j,
n+1}\in (s_{j, n+1}, r_{j, n+1}),$$ inequality \eqref{L2.3eq13} and limit
relation \eqref{L2.3eq4} imply that
\begin{equation}\label{L2.3eq14}
\lim_{n\to\infty}\frac{r_{j, n}}{s_{j, n+1}}=\infty.
\end{equation}
Hence there is $m_1\in\mathbb N$ such that
\begin{equation}\label{L2.3eq15}
\frac{r_{j, n}}{s_{j, n+1}}\ge q^{3(N+1)}
\end{equation}
holds for all $n\in \mathbb N\setminus \textbf{\emph{m}}_{\textbf{1}}$ and
$j\in \textbf{\emph{N}}.$ Using \eqref{L2.3eq13} and \eqref{L2.3eq15}, we see, in
particular, that
\begin{equation}\label{L2.3eq16}
(r_{j,n_1}, s_{j,n_1})\cap (r_{j, n_2}, s_{j, n_2})=\varnothing
\end{equation}
if $n_1, n_2 \in\mathbb N\setminus \textbf{\emph{m}}_{\textbf{1}},$ $n_1 \ne
n_2$ and $j\in \textbf{\emph{N}}.$ This disjointness together with \eqref{L2.3eq9} and
\eqref{L2.3eq12} yield
\begin{equation}\label{L2.3eq17}
B_{j}(q)=\bigcup_{n=0}^{\infty}B_{j, n}(q)=\bigcup_{n=
m_{1}+1}^{\infty}(r_{j, n}, s_{j, n})\cup O_{j, q,
m_1}
\end{equation}
for every $j\in \textbf{\emph{N}}$ with $O_{j, q,
m_1}:=B_{j}(q)\cap [r_{j, m_1},
\infty).$

Obviously, for every $x\in E(q)$ there is a unique connected component $(a_x,
b_x),$ $a_x =a_{x}(q)$ and $b_x =b_{x}(q),$ of the set $E(q)$ such that $x\in (a_x, b_x).$ As is easily seen the following statements are valid:

\medskip

\noindent$\bullet$ The chain $(\mathrm{Cc}^{1}E(q), \preceq)$ is infinite if
there is $t\in (0, \infty)$ such that $a_x >0$ for every $x\in (0,
t)\cap E(q);$
\medskip

\noindent$\bullet$ Inequality \eqref{L2.3eq1} holds if there are $t\in (0,
\infty), k\in (1, \infty)$ and $p\in\mathbb N$ such that
\begin{equation}\label{L2.3eq18}
k^{-p}x < a_x
\end{equation}
for every $x\in (0, t)\cap E(q).$

Note also that the inequalities $q_1 \ge q_2 >1$ imply  the
inclusion $E(q_1)\supseteq E(q_2).$ Thus, the inclusion $(a_x (q_1),
b_x (q_1))\supseteq (a_x (q_2), b_x (q_2))$ holds if $q_1 \ge q_2
>1.$
Consequently, to prove the first part of the lemma it is sufficient
to show that \eqref{L2.3eq18} holds if
\begin{equation}\label{L2.3eq19}
q\ge \bigvee_{j=1}^{N}q_{j}^{2} \quad\mbox{and}\quad x\in (0,
r^{1})\cap E(q)
\end{equation}
where \begin{equation}\label{n(eq)}r^{1}:=\mathop{\bigwedge}\limits_{j=1}^{N}r_{j, m_1}.\end{equation} Let $x\in\mathbb R^{+}.$ To find
$k\in (1, \infty)$ and $p\in\mathbb N,$ which satisfy \eqref{L2.3eq18}, we
define a subset $\textbf{\emph{J}}_{x}$ of $\textbf{\emph{N}}$ by the
rule:
\begin{equation}\label{rule}
(j\in \textbf{\emph{J}}_{x})\Leftrightarrow (j\in \textbf{\emph{N}}
\quad \mbox{and} \quad x\in (0, r^{1})\cap B_{j}(q)),
\end{equation}
where $r^{1}$ is defined in \eqref{n(eq)}. From \eqref{rule} it is clear
that
\begin{equation}\label{rule1}
(\textbf{\emph{J}}_{x}=\varnothing)\Leftrightarrow (x\in [r^1,
\infty) \quad\mbox{or}\quad x\in\mathbb R^{+}\setminus E(q)).
\end{equation}
Let \eqref{n(eq)} hold and let
\begin{equation}\label{L2.3eq20}
\theta\in (q^{3}, q^{3(N+1)}).
\end{equation}
We claim that if $\textbf{\emph{J}}_x\ne\varnothing\ne
\textbf{\emph{J}}_{\theta^{-1}x},$ then the equality
\begin{equation}\label{L2.3eq21}
\textbf{\emph{J}}_x\cap \textbf{\emph{J}}_{\theta^{-1}x}=\varnothing
\end{equation}
holds. Suppose contrary, that $\textbf{\emph{J}}_x\ne\varnothing\ne\textbf{\emph{J}}_{\theta^{-1}x}$ holds, but there is $j_{0}\in\textbf{\emph{N}}$ such that $j_{0}\in\textbf{\emph{J}}_x\cap\textbf{\emph{J}}_{\theta^{-1}x}.$ Then, using \eqref{rule}, we see that there are $n_1, n_2\in \mathbb N\setminus \textbf{\emph{m}}_{\textbf{1}},$ so that
\begin{equation}\label{L2.3eq22}
x\in (r_{j_0, n_2}, s_{j_0, n_2}) \quad \mbox{and} \quad \theta^{-1}x\in (r_{j_0, n_1}, s_{j_0, n_1}).
\end{equation}
If $n_1 = n_2,$ then the inequalities $r_{j_0, n_1}<\theta^{-1}x<x<s_{j_0, n_1}$ hold. Hence, we have
\begin{equation*}
\theta=\frac{x}{\theta^{-1}x}\le\frac{s_{j_0, n_1}}{r_{j_0, n_1}}.
\end{equation*}
Now, using \eqref{L2.3eq20}, we obtain
\begin{equation*}
q^3 <\theta\le\frac{s_{j_0, n_1}}{r_{j_0, n_1}},
\end{equation*}
contrary to \eqref{L2.3eq13}. Hence, $n_1\ne n_2.$ The relations $\theta^{-1}x <x$ and $n_1\ne n_2$ imply the inequality $n_1 >n_2.$ Consequently, $n_2 < n_{2}+1\le n_1.$ These inequalities and \eqref{z2}, imply
\begin{equation*}
(r_{j_0, n_2}, s_{j_0, n_2})\prec(r_{j_0, n_{2}+1}, s_{j_0, n_{2}+1})\preceq (r_{j_0, n_1}, s_{j_0, n_1}).
\end{equation*}
Hence,
\begin{equation}\label{L2.3eq23}
\theta=\frac{x}{\theta^{-1}x}\ge\frac{r_{j_0, n_{1}+1}}{s_{j_0, n_{1}}}.
\end{equation}
From \eqref{L2.3eq20} and \eqref{L2.3eq23} it follows that
\begin{equation*}
q^{3(N+1)}>\frac{r_{j_0, n_{1}+1}}{s_{j_0, n_1}},
\end{equation*}
contrary to \eqref{L2.3eq15}. Thus, \eqref{L2.3eq21} holds, if $\textbf{\emph{J}}_x\ne\varnothing$ and $\textbf{\emph{J}}_{\theta^{-1}x}\ne\varnothing.$

Now, let $k\in (q^3, q^{\frac{3(N+1)}{N}}).$ It is simple to prove that
\begin{equation*}
q^3 <k< ...<k^{N}<q^{3(N+1)}.
\end{equation*}
Hence \eqref{L2.3eq20} holds, if $\theta=k^{m}$ and $m\in \textbf{\emph{N}}.$ Consequently, if we have
\begin{equation}\label{L2.3eq24}
\textbf{\emph{J}}_{k^{-m}x}\ne\varnothing
\end{equation}
for every $m\in \textbf{\emph{N}}\cup\{0\},$
then
\begin{equation}\label{L2.3eq25}
\textbf{\emph{J}}_{k^{-m_1}x}\cap\textbf{\emph{J}}_{k^{-m_2}x}=\varnothing
\end{equation}
for all distinct $m_1, m_2\in\textbf{\emph{N}}\cup\{0\}.$ (To see it suppose $m_1 < m_2$ and replace in \eqref{L2.3eq20} $x$ and $\theta^{-1}x$ by $k^{-m_1}x$ and $k^{-(m_2-m_1)}k^{-m_1}x$ respectively.)
By \eqref{L2.3eq25}, $\textbf{\emph{J}}_x, \textbf{\emph{J}}_{k^{-1}x}, ...\textbf{\emph{J}}_{k^{-N}x}$
are disjoint subsets of $\textbf{\emph{N}}.$ Hence, if \eqref{L2.3eq24} holds, then
\begin{equation}\label{L2.3eq26}
N=|\textbf{\emph{N}}|\ge\sum_{l=0}^{N}|\textbf{\emph{J}}_{k^{-l}x}|\ge\sum_{l=0}^{N}1=N+1.
\end{equation}
This contradiction shows, that there is $l\in \textbf{\emph{N}}\cup\{0\}$ such that
\begin{equation}\label{L2.3eq27}
\textbf{\emph{J}}_{k^{-l}x}=\varnothing.
\end{equation}
Assume that $x\in (0, r^1)\cap E(q).$ By \eqref{rule}, equality \eqref{L2.3eq27} holds if and only if
\begin{equation*}
k^{-l}x\in[r^1, \infty) \quad\mbox{or}\quad k^{-l}x\in\mathbb R^{+}\setminus E(q).
\end{equation*}
Since $0<k^{-l}x<x<r^1,$ \eqref{L2.3eq27} yields that $k^{-l}x\not\in E(q).$ Since $(a_x, b_x)$ is a connected component of the set $E(q),$ it is proved, that the inequality
\begin{equation}\label{L2.3eq28}
k^{-N}x<a_x
\end{equation}
holds whenever $x\in (a_x, b_x)\in \mathrm{Cc}^{1} E(q), \, x<r^1$ and $q\ge\mathop{\bigvee}\limits_{j=1}^{N}q_{j}^{2}.$ Since $(\mathrm{Cc}^{1}E(q), \preceq)$ is infinite for every $q>1,$ the assertion \eqref{L2.3eq1} holds for
\begin{equation}\label{L2.3eqv}
q>q_{0}:=\bigvee_{j=1}^{N}q_{j}^{2}.
\end{equation}
To complete the proof it suffices to show that \eqref{L2.3eq2} holds with $M=N.$

Let \eqref{L2.3eqv} hold and let
\begin{equation*}
(a_i, b_i)\in\{(a_n, b_n)\}_{n\in\mathbb N}=\mathrm{Cc}^{1}E(q).
\end{equation*}
For $i\in\mathbb N$ define a set $\textbf{\emph{J}}_{i}\subseteq \textbf{\emph{N}}$ as
\begin{equation}\label{L2.3eq29}
\textbf{\emph{J}}_{i}:=\bigcup_{x\in(a_i, b_i)}\textbf{\emph{J}}_{x}
\end{equation}
where $\textbf{\emph{J}}_{x}$ was defined by \eqref{rule}. It follows from \eqref{L2.3eq29} and \eqref{rule1} that there is $i_{0}\in\mathbb N$ such that $\textbf{\emph{J}}_{i}\ne\varnothing$ for $i\ge i_{0},$ i. e. $$(a_i, b_i)\cap (0, r^{1})\cap E(q)\ne\varnothing,$$ for $i\ge i_0.$ Hence, without loss of generality, we may suppose that if $x\in (a_i, b_i)$ and $i\ge i_0,$ then $x< r^1.$ Consequently, for every $i\ge i_0$ there is $l\in \textbf{\emph{N}}$ such that
\begin{equation}\label{L2.3eq30}
\emph{\textbf{J}}_{i}\cap\emph{\textbf{J}}_{i+l}\ne\varnothing.
\end{equation}
In the opposite case, the sets $\emph{\textbf{J}}_{l}, \emph{\textbf{J}}_{l+1}, ..., \emph{\textbf{J}}_{i+N}$ are disjoint nonempty subsets of $\textbf{\emph{N}},$ which contradicts the equality $|\textbf{\emph{N}}|=N$. If \eqref{L2.3eqv} holds, then there are $y_{i}\in (a_i, b_i)$ and $y_{i+l}\in (a_{i+l}, b_{i+l})$ such that $\emph{\textbf{J}}_{x_i}\cap\emph{\textbf{J}}_{x_{i+l}}\ne\varnothing.$ Let $j_1\in \emph{\textbf{J}}_{x_i}\cap\emph{\textbf{J}}_{x_{i+l}}.$ Then we have $$y_i, y_{i+l}\in B_{j_1}(q).$$ Using \eqref{L2.3eq17}, we can find $(r_{j_1, n_1}, s_{j_1, n_1})$ and $(r_{j_1, n_2}, s_{j_1, n_2})$ such that $n_1 >n_2,$ $$y_{i+l}\in(r_{j_1, n_1}, s_{j_1, n_1})\quad \mbox{and} \quad y_i\in(r_{j_1, n_2}, s_{j_1, n_2}).$$ Indeed, if $n_1=n_2,$ then the points $y_i$ and $y_{i+l}$ belong to one and the same connected component of $E(q).$ Using \eqref{L2.3eq14}, we can show that
\begin{equation}\label{L2.3eq31}
\lim_{i\to\infty}\frac{y_i}{y_{i+l}}=\infty.
\end{equation}
Note also that if $b_j<r^1$ and $q\ge\mathop{\bigvee}\limits_{j=1}^{N}q_{j}^{2},$ then, using \eqref{L2.3eq28}, we can prove that
\begin{equation}\label{L2.3eq32}
k^{-N}\le\frac{a_i}{b_j}\quad\mbox{for $k\in(q^3, q^{\frac{3(N+1)}{N}}).$}
\end{equation}
Now \eqref{L2.3eq31}, \eqref{L2.3eq32} and condition $l\in \textbf{\emph{N}}$ imply \eqref{L2.3eq2} with $M=N.$
\end{proof}

Using Lemma~\ref{L16} and Lemma~\ref{L2.3}, we obtain the following blow up description of the ideal $I(\textbf{\emph{CSP}}).$
\begin{theorem}\label{Th2.19}
Let $E\subseteq\mathbb R^{+}$ and $0\in \mathrm{ac} E.$ Then the following conditions are equivalent:
\begin{enumerate}
\item[\rm(i)]\textit{$E\in I(\textbf{CSP});$}

\item[\rm(ii)]\textit{the chain $\mathrm{Cc}^{1}E(q)$ is infinite for every $q>1,$ $\mathrm{Cc}^{1} E(q)=\{(a_i, b_i)\}_{i\in\mathbb N},$ and there are $q_{0}>1$ and $M\in\mathbb N$ such that
\begin{equation*}
\limsup_{i\to\infty}\frac{b_i}{a_i}<\infty\quad \mbox{and}\quad \lim_{n\to\infty}\bigvee_{j=0}^{M}\frac{a_{n+j}}{b_{n+j+1}}=\infty \quad\mbox{for all $q>q_0.$}
\end{equation*}
}
\end{enumerate}
\end{theorem}

Theorem~\ref{Th2.19} and Theorem~\ref{Th2.20} imply the following corollary.
\begin{corollary}\label{Col(result)}
We have the inclusion $I(\textbf{CSP})\subseteq \hat I(\textbf{SP}).$
\end{corollary}

The following example shows that there exists a set $E\subseteq\mathbb R^{+}$ such that $E\in\hat{I} (\textbf{\emph{SP}})$ but $E\not\in I (\textbf{\emph{CSP}}).$

\begin{example}\label{ex(result)}
Let $\alpha\in (0, 1).$ For every $j\in\mathbb N$ define positive numbers $y_{0, j}, y_{1, j}, ..., y_{j, j}$ so that
\begin{equation*}
y_{1, j}=\alpha^{1}y_{0, j}, \, y_{2, j}=\alpha^{2}y_{1, j}, \, ...,\, y_{j, j}=\alpha^{j}y_{j-1, j} \quad\mbox{and}\quad y_{0, j+1}< y_{j, j},
\end{equation*}
and
\begin{equation*}
\lim_{j\to\infty}\frac{y_{j, j}}{y_{0, j+1}}=\infty.
\end{equation*}
Write
\begin{equation*}
E=\bigcup_{j\in\mathbb N}\left(\bigcup_{k=0}^{j}\{y_{k, j}\}\right).
\end{equation*}
Let $q>1.$ Simple estimations show that $\mathrm{Cc}^{1} E(q)$ is infinite, $\mathrm{Cc}^{1} E(q)=\{(a_i, b_i)\}_{i\in\mathbb N}$ and
\begin{equation*}
\limsup_{i\to\infty}\frac{b_i}{a_i}\le\left(\frac{1}{\alpha}\right)^{m}+\left(\frac{1}{\alpha}\right)^{m-1}+ ... +\frac{1}{\alpha}+1,
\end{equation*}
where $m$ is the smallest positive whole number such that
\begin{equation}\label{exeq}
q<\left(\frac{1}{\alpha}\right)^{m}.
\end{equation}
Consequently, by Theorem~\ref{Th2.20}, we have
\begin{equation*}
E\in\hat I (\textbf{\emph{SP}}).
\end{equation*}
In accordance with Theorem~\ref{Th2.19}, the statement $E\in I (\textbf{\emph{CSP}})$ does not hold if and only if the inequality
\begin{equation*}
\liminf_{n\to\infty}\bigvee_{j=0}^{M}\frac{a_{n+j}}{b_{n+j+1}}<\infty
\end{equation*}
holds for every $q>1$ and $M\in\mathbb N.$
Let $m\in\mathbb N$ satisfy \eqref{exeq}. Then we can show that
\begin{equation*}
\liminf_{n\to\infty}\bigvee_{j=0}^{M}\frac{a_{n+j}}{b_{n+j+1}}\le\left(\frac{1}{\alpha}\right)^{m+M+1}.
\end{equation*}
Thus, $E$ does not belong to $I (\textbf{\emph{CSP}}).$
\end{example}

\medskip

\noindent\textbf{Acknowledgment.} The research of the second author was supported as a part of EUMLS project with grant agreement PIRSES$-$GA$-$2011$-$295164.

\textbf{Viktoriia Bilet}

Institute of Applied Mathematics and Mechanics of NASU, R. Luxemburg str. 74, Donetsk 83114, Ukraine

\textbf{E-mail:} biletvictoriya@mail.ru

\bigskip

\textbf{Oleksiy Dovgoshey}

Institute of Applied Mathematics and Mechanics of NASU, R. Luxemburg str. 74, Donetsk 83114, Ukraine

\textbf{E-mail:} aleksdov@mail.ru

\bigskip

\textbf{J\"{u}rgen Prestin}

Institut fuer Mathematik, Universit\"{a}t zu L\"{u}beck, Ratzeburger Allee 160, D-23562 L\"{u}beck, Germany

\textbf{E-mail:} prestin@math.uni-luebeck.de

\end{document}